\documentclass[reqno, 11pt]{amsart}
\usepackage[mathscr]{eucal}
\usepackage[all]{xy}
\usepackage{mathrsfs}
\usepackage{xypic}
\usepackage{amsfonts}
\usepackage{amsmath}
\usepackage{amsthm,bm}
\usepackage{amssymb,enumerate}
\usepackage{latexsym}
\usepackage{tabularx}
\usepackage{graphicx}
\usepackage{pict2e}
\usepackage{tikz}
\usepackage[pagebackref]{hyperref}
\usepackage[text={160mm,240mm},centering]{geometry}
\input diagxy
\input xy

\newtheorem{thm}{Theorem}[section]
\newtheorem{mthm}[thm]{Main Theorem}

\newtheorem{lem}[thm]{Lemma}
\newtheorem{prop}[thm]{Proposition}

\theoremstyle{definition}
\newtheorem{defn}[thm]{Definition}

\numberwithin{equation}{section}

\newcommand{\im}{\textmd{im}}
\begin{document}

\title{On the Morse--Novikov Cohomology of blowing up complex manifolds}

\author{Yongpan Zou}
\address{School of Mathematics and Statistics, Wuhan University, Wuhan 430072, P. R. China}
\email{yongpan\_zou@whu.edu.cn}%


\date{\today}


\begin{abstract}
Inspired by the recent works of S. Rao--S. Yang--X.-D. Yang and L. Meng on the blow-up formulae for de Rham and Morse--Novikov cohomology groups,
we give a new simple proof of the blow-up formula for Morse--Novikov cohomology by introducing the relative Morse--Novikov cohomology group via sheaf  cohomology theory and presenting the explicit isomorphism therein.
\end{abstract}

\maketitle
\setcounter{tocdepth}{1}


\section{Introduction}
Let $X$ be a  smooth manifold and $\mathcal{A}^p(X)$ the space of smooth $p$-forms on $X$. Over $X$, there is an exterior differential operator $d$ and the associated complex:
\begin{displaymath}
\xymatrix{
\cdots\ar[r] &\mathcal{A}^{p-1}(X)\ar[r]^{d}&\mathcal{A}^{p}(X)\ar[r]^{d}&\mathcal{A}^{p+1}(X)\ar[r]&\cdots
,}
\end{displaymath}
whose cohomology $H^p(X) := H^p(\mathcal{A}^\bullet(X),d)$ is called \emph{$p$-th de Rham cohomology group}.
Now we choose a closed $1$-form $\theta$ on $X$. For $\alpha\in\mathcal{A}^p(X)$, define $$d_\theta:\mathcal{A}^p(X)\rightarrow\mathcal{A}^{p+1}(X)$$ as $$d_\theta\alpha=d\alpha+\theta\wedge\alpha.$$
Obviously, $d_\theta\circ d_\theta=0$ and we have a complex:
\begin{displaymath}
\xymatrix{
\cdots\ar[r] &\mathcal{A}^{p-1}(X)\ar[r]^{d_\theta}&\mathcal{A}^{p}(X)\ar[r]^{d_\theta}&\mathcal{A}^{p+1}(X)\ar[r]&\cdots
,}
\end{displaymath}
whose cohomology $H^p_\theta(X) := H^p(\mathcal{A}^\bullet(X),d_\theta)$ is called  \emph{$p$-th Morse--Novikov cohomology group}.

This cohomology was originally defined by Lichnerowicz \cite{GL,L} and Novikov \cite{N} in the context of Poisson geometry and Hamiltonian mechanics, respectively. It is commonly used to study the locally conformally K\"{a}hlerian and locally conformally symplectic structures. It's not a topological invariant but a conformal invariant. One significant application of locally conformally K\"{a}hlerian structure in complex geometry is the classification of compact non-K\"{a}hler complex surfaces. Here we refer to \cite{Be} and \cite{ovv} for more references therein.

For a complex manifold $X$ and a submanifold $Z \subseteq X$, we consider a new complex manifold, the blow-up $\widetilde{X}$ of $X$ along $Z$. It's interesting to find the relations between various cohomologies of $\widetilde{X}$ with those of $X$ and $Z$.
The blow-up formula for de Rham cohomology has been presented in \cite[Theorem 7.31]{V}. In \cite{ryy}, S. Rao--S. Yang--X.-D. Yang give a new proof of the blow-up formula for de Rham cohomology by use of the relative de Rham cohomology. In \cite{YZ15}, X.-D. Yang--G. Zhao prove a blow-up formula of Morse--Novikov cohomology for compact locally conformal K\"ahler manifolds with some additional restriction on the closed $1$-forms $\theta$. In \cite{meng2}, L. Meng systematically studies the behavior of Morse--Novikov cohomology under blow-up along a connected complex submanifold $Z$. Especially, he defines the weight $\theta$-sheaf $\underline{\mathbb{R}}_{X,\theta}$ and reinterprets Morse--Novikov cohomology via sheaf theory. Meng also establishes a theorem of Leray--Hirsch type for Morse--Novikov cohomology to prove the blow-up formulae with explicit isomorphisms.

In this paper, we combine the relative cohomological method in \cite{ryy,ryy2,ryyy} with the sheaf theory of \cite{meng2}, to give a new simple proof of the blow-up formula for Morse--Novikov cohomology with an explicit isomorphism.
\begin{mthm} \label{main-thm}
Let $X$ be a complex manifold with $\emph{dim}_{\mathbb{C}}\,X=n$, $Z\subseteq X$ a closed complex submanifold of complex codimension $r\geq2$ and $i^{\ast}$ the pullback of the inclusion $i:Z\hookrightarrow X$. Suppose that $\pi:\tilde{X}\rightarrow X$ is the blow-up of $X$ along $Z$. Denote by $E:=\pi^{-1}(Z)\cong \mathbb{P}(\mathcal{N}_{Z/X})$ the exceptional divisor of the blow-up. Set $\tilde{i}^\ast$ as the pullback of the inclusion $\tilde{i}: E \hookrightarrow \tilde{X}$. Then for any $0\leq k\leq 2n$, the map
$$
 \phi=\pi_{\ast}+\sum_{j=1}^{r-1}\Pi_{j}\circ \tilde{i}^{\ast}
$$
gives isomorphism
$$\label{1.1}
  H^{k}_{\pi^\ast\theta}(\tilde{X}) \cong  H^{k}_{\theta}(X)\oplus \Big(\bigoplus_{j=1}^{r-1} H^{k-2j}_{i^\ast\theta}(Z) \Big),
$$
where the definition of $ \phi$ is given in \eqref{phi}.
\end{mthm}
This blow-up formula was first proved by Meng in \cite[Main Theorem 1.3]{meng2}. Meng's explicit isomorphism is
 \begin{equation}\label{psi}
 \psi:= \pi^*+\sum_{j=1}^{r-1}(\tilde{i})_*\circ (h^j\cup)\circ (\pi|_E)^*,
 \end{equation}
where $\pi|_E: E:=\pi^{-1}(Z)\rightarrow Z$ is the projection of the projectivization $E\cong\mathbb{P}(N_{Z/X})$ of the normal bundle $N_{Z/X}$, and $h:=c_1(\mathcal{O}_E(-1))\in H_{\textrm{dR}}^2(E)$ is the first Chern class of the universal line bundle $\mathcal{O}_E(-1)$. This morphism maps the cohomology group $ H^{k}_{\theta}(X)\oplus \Big(\bigoplus_{j=1}^{r-1} H^{k-2j}_{i^\ast\theta}(Z) \Big)$ to $H^{k}_{\pi^\ast\theta}(\tilde{X})$. The explicit morphism $\phi$ constructed here is inspired by \cite{ryy2}, and owns an inverse direction to Meng's one. Actually in \cite{meng3}, Meng points out that these two morphism inverse to each other.

Moreover, the relative de Rham cohomology is isomorphic to the de Rham cohomology with compact support, while it does not hold true for the Morse--Novikov case anymore (see \cite{ryy,z19}). So we are not able to apply the relative cohomological method via the compactly supported cohomology  directly as \cite{ryy,ryy2}.

\subsection*{Acknowledgement}
The author is indebted to Professor Sheng Rao for his constant support and for many useful discussions. We also would like to thank Professor Lingxu Meng for many useful discussions, especially the sheaf theory in subsection \ref{sheaf 2.3}, and Professor X.-D. Yang for the information on the background of Morse--Novikov cohomology.

\section{Preliminaries}\label{pre}

\subsection{The sequence associated to a closed submanifold}

Assume that $X$ is a smooth manifold with dimension $n$ and let $Z$ be a $k$-dimensional closed submanifold of $X$.
In this paper we focus on the space of differential forms
$$
\mathcal{A}^{\bullet}(X,Z)=\{\alpha\in\mathcal{A}^{\bullet}(X): i^{\ast}\alpha=0\}.
$$
For any closed $1$-form $\theta$ on $X$, $i^{\ast}\theta$ is a closed $1$-form on $Z$. It's easy to see that for any $\alpha\in\mathcal{A}^{\bullet}(X) $,
$$ i^\ast d_\theta\alpha = i^\ast(d\alpha+\theta\wedge\alpha) = d(i^\ast\alpha)+i^\ast\theta\wedge i^\ast\alpha = d_{i^\ast\theta}i^\ast\alpha. $$
Therefore,
 $$ \alpha\in\mathcal{A}^{\bullet}(X,Z) \Rightarrow i^\ast\alpha=0 \Rightarrow i^\ast d_\theta\alpha=0 \Rightarrow d_\theta\alpha\in \mathcal{A}^{\bullet}(X,Z).$$
So $\mathcal{A}^{\bullet}(X,Z)$ is closed under the action of the exterior differential operator $d_{\theta}$ and we get a sub-complex of the Morse--Novikov complex $\{\mathcal{A}^{\bullet}(X),d_\theta\}$, called the \emph{relative Morse--Novikov complex} with respect to $Z$:
$$
\xymatrix{
0 \ar[r]^{} & \mathcal{A}^{0}(X,Z) \ar[r]^{d_\theta} & \mathcal{A}^{1}(X,Z)  \ar[r]^{d_\theta} & \mathcal{A}^{2}(X,Z) \ar[r]^{\quad d_\theta} & \cdots.}
$$
The associated cohomology, denoted by $H_{\theta}^{\bullet}(X,Z)$, is called the \emph{relative Morse--Novikov cohomology} of the pair $(X,Z)$.

\begin{lem}\label{lemma 2.1}
The pullback $i^{\ast}:\mathcal{A}^{\ast}({X})\rightarrow \mathcal{A}^{\ast}({Z})$ is surjective.
\end{lem}
\begin{proof}
By the classical tubular neighborhood theorem \cite[Theorem 6.24]{lee}, we have an open tubular neighborhood $V$ of $Z$ with a smooth retraction map $\gamma: V\rightarrow Z$ such that $\gamma|_{Z}$ is the identity map of $Z$ as in \cite[Proposition 6.25]{lee}. As usual, we let $i: Z\hookrightarrow V$ be the inclusion. Let $ \varphi: X\rightarrow[0,1]$ be a smooth cut-off function such that $Supp(\varphi)\subseteq V $ and such that $ \varphi $ is constantly equal to $1$ on some open set $ W \subseteq V $ with $ Z \subseteq W $. For any $q$-forms $ \alpha\in \mathcal{A}^{q}(Z)$, we define $ \widetilde{\alpha}:= \gamma^{\ast}(\alpha)\in \mathcal{A}^{q}(V)$ and $\tau:= \varphi\widetilde{\alpha}$. It is easy to see that $\tau$ extends trivially over the remaining part of $X$, that is $\tau \in \mathcal{A}^{q}(X)$. Therefore we can see $ i^{\ast}\tau=i^{\ast}\varphi\gamma^{\ast}(\alpha)=\alpha.$
\end{proof}

\subsection{Blow-up}
Let $X$ be a complex manifold of complex dimension $n$. Suppose that $i: Z \hookrightarrow X$ is a closed complex submanifold of complex codimension $r\geq2$. Without loss of generality, we assume that $Z$ is connected; otherwise, we can carry out the blow-up operation along each connected component of $Z$ step by step. Recall that the \emph{normal bundle} $T_{X|Z}/T_Z$ of $Z$ in $X$, denoted by $\mathcal{N}_{Z/X}$, is a holomorphic vector bundle of rank $r$. Here $T_M$ denotes the holomorphic tangent bundle of the complex manifold $M$ and $T_{M|N}$ is its restriction to the submanifold $N$ of $M$.      The \emph{blow-up $\tilde{X}$ of $X$ with center $Z$} is a projective morphism  $\pi: \tilde{X}\to X$ such that
$$
\pi: \tilde{X}-E \to X-Z
$$
is a biholomorphism.
Here
\begin{equation} \label{excep}
E:=\pi^{-1}(Z)\cong \mathbb{P}(\mathcal{N}_{Z/X})
\end{equation}
is the \emph{exceptional divisor} of the blow-up.
Then one has the following blow-up diagram
\begin{equation}\label{blow-up}
\xymatrix{
E \ar[d]_{\pi_{E}} \ar@{^{(}->}[r]^{\tilde{i}} & \tilde{X}\ar[d]^{\pi}\\
 Z \ar@{^{(}->}[r]^{i} & X.
}
\end{equation}

\subsection{Morse-Novikov cohomology via sheaf theory} \label{sheaf 2.3}
The notations and definitions follow \cite{meng2}. Suppose that $ \mathcal{A}_X^k $ is the sheaf of germs of smooth $k$-forms, and we call the kernel of $ d_\theta : \mathcal{A}_X^0 \rightarrow \mathcal{A}_X^1 $ a \emph{weight $\theta$-sheaf}, denoted by $ \underline{\mathbb{R}}_{X,\theta} $. The weight $\theta$-sheaf $\underline{\mathbb{R}}_{X,\theta} $ is a locally constant sheaf of $ \mathbb{R}$-modules of rank $1$. We have a resolution of soft sheaves of $ \underline{\mathbb{R}}_{X,\theta}, $
\begin{center}
$ \xymatrix@C=0.5cm{
  0 \ar[rr]^{} && \underline{\mathbb{R}}_{X,\theta} \ar[rr]^{i} && \mathcal{A}_X^0 \ar[rr]^{d_\theta} && \mathcal{A}_X^1 \ar[rr]^{d_\theta} && \cdot\cdot\cdot \ar[rr]^{d_\theta} && \mathcal{A}_X^n \ar[rr]^{} && 0,}
$
\end{center}
where $ i $ is the inclusion.
Let $ \mathcal{A}_X^\bullet \rightarrow \mathcal{I}^\bullet $ be an injective resolution of complex $ (\mathcal{A}_X^\bullet,d_\theta) $ of  sheaves in the category of sheaves on $ X $ and  it induces a morphism
  $$ H_\theta^\ast(X) = H^\ast(\mathcal{A}^\bullet(X),d_\theta) \rightarrow H^\ast(\Gamma(X,\mathcal{I}^\bullet)) = H^\ast(X,\mathbb{R}_{X,\theta}), $$
denoted by $ \rho_{X,\theta} $. Since $ \mathcal{A}_X^\bullet $ is a resolution of soft sheaves of $ \underline{\mathbb{R}}_{X,\theta} $, $ \rho_{X,\theta} $ is an isomorphism.

The next lemma plays an important role in our paper.
\begin{lem}[{\cite[Lemma 2.2]{meng2}}] \label{lemma 2.2}
Let $X$ be a connected smooth manifold and $\theta$ a closed $1$-form on $X$. Suppose that $f:Y\rightarrow X$ is a smooth map between smooth manifolds.  Then the inverse image sheaf $f^{-1}\underline{\mathbb{R}}_{X,\theta}\cong\underline{\mathbb{R}}_{Y,f^*\theta}$.
\end{lem}

Let $i: Z \rightarrow X$ be a closed submanifold and $\theta$ a closed $1$-form on $X$. We know that $i^{\ast}: \mathcal{A}_X^{p} \rightarrow i_{\ast}\mathcal{A}_Z^{p}$ is an epimorphism of sheaves.  Denote the kernel of $i^{\ast}$ by $ \mathcal{A}_{X,Z}^{p}. $
Thus one gets an exact sequence of sheaves
$$
\xymatrix{
0\ar[r] &\mathcal{A}_{X,Z}^{p}\ar[r]&\mathcal{A}^{p}_{X}\ar[r]^{i^{\ast}}&i_{\ast}\mathcal{A}_Z^{p}\ar[r]&0
.}
$$
Then there is a short exact sequence of complexes of sheaves
\begin{equation} \label{2.4}
\xymatrix{
0\ar[r] &\mathcal{A}_{X,Z}^{\bullet}\ar[r]&\mathcal{A}^{\bullet}_{X}\ar[r]^{i^{\ast}}&i_{\ast}\mathcal{A}_Z^{\bullet}\ar[r]&0
}.
\end{equation}
We now consider the complexes of sheaves:
$$
\xymatrix{
0 \ar[r]^{} & \underline{\mathbb{R}}_{X,\theta} \ar[r] & \mathcal{A}^{0}_{X}  \ar[r]^{d_\theta} & \mathcal{A}^{1}_{X} \ar[r]^{d_\theta} & \cdots,}
$$

$$
\xymatrix{
0 \ar[r]^{} & \underline{\mathbb{R}}_{Z,i^{\ast}\theta} \ar[r] & \mathcal{A}^{0}_{Z}  \ar[r]^{d_{i^{\ast}\theta}} & \mathcal{A}^{1}_{Z} \ar[r]^{d_{i^{\ast}\theta}} & \cdots,}
$$

$$
\xymatrix{
0 \ar[r]^{} & \underline{\mathbb{R}}_{X,Z,\theta} \ar[r] & \mathcal{A}^{0}_{X,Z}  \ar[r]^{d_\theta} & \mathcal{A}^{1}_{X,Z} \ar[r]^{d_\theta} & \cdots,}
$$
where
$$ \underline{\mathbb{R}}_{Z,i^{\ast}\theta} := \ker\ (d_{i^{\ast}\theta} : \mathcal{A}_Z^0 \rightarrow \mathcal{A}_Z^1),$$
$$\underline{\mathbb{R}}_{X,Z,\theta} := \ker\ (d_\theta : \mathcal{A}_{X,Z}^0 \rightarrow \mathcal{A}_{X,Z}^1) .$$

For the complex $\mathfrak{F}^{\bullet}$ of sheaves
 \begin{displaymath}
\xymatrix{
0\ar[r] &\mathfrak{F}^{p-1}\ar[r]&\mathfrak{F}^{p}\ar[r]^{}&\mathfrak{F}^{p+1}\ar[r]&\cdot\cdot\cdot
},
\end{displaymath}
we define the cohomology of the complex $\mathfrak{F}^{\bullet}$ as
$$\mathcal{H}^{p}(\mathfrak{F}^{\bullet}) := \frac{\ker\ (\mathfrak{F}^{p} \rightarrow \mathfrak{F}^{p+1})}{\im\ (\mathfrak{F}^{p-1} \rightarrow \mathfrak{F}^{p})}.
$$
Then \eqref{2.4} yields a long exact sequence of sheaves
$$
\xymatrix{
\cdots\ar[r]^{} & \mathcal{H}^{p-1}(i_{\ast}\mathcal{A}_Z^{\bullet}) \ar[r]^{} & \mathcal{H}^{p}(\mathcal{A}_{X,Z}^{\bullet}) \ar[r]^{} & \mathcal{H}^{p}(\mathcal{A}_{X}^{\bullet})\ar[r]^{} & \mathcal{H}^{p}(i_{\ast}\mathcal{A}_Z^{\bullet}) \ar[r]^{} & \cdots.}
$$
It's easy to see
\begin{equation*}
\mathcal{H}^{p}(\mathcal{A}_{X}^{\bullet}) = \frac{\ker\ (\mathcal{A}_{X}^{p} \rightarrow \mathcal{A}_X^{p+1})}{\im\ (\mathcal{A}_{X}^{p-1} \rightarrow \mathcal{A}_X^{p})} =
 \begin{cases}
\underline{\mathbb{R}}_{X,\theta}, &p=0, \\
 0, &p\neq 0;
 \end{cases}
\end{equation*}

\begin{equation*}
\mathcal{H}^{p}(i_{\ast}\mathcal{A}_Z^{\bullet}) = i_{\ast}\mathcal{H}^{p}(\mathcal{A}_Z^{\bullet}) =
\begin{cases}
   i_{\ast}\underline{\mathbb{R}}_{Z,i^{\ast}\theta}, & p=0, \\
 0, & p\neq 0.
\end{cases}
\end{equation*}
So the long exact sequence actually turns into
$$
\xymatrix{
0 \ar[r]^{} & \underline{\mathbb{R}}_{X,Z,\theta} \ar[r]& \underline{\mathbb{R}}_{X,\theta} \ar[r] &i_{\ast}\underline{\mathbb{R}}_{Z,i^{\ast}\theta}\\
\ar[r]& \mathcal{H}^{0}(\mathcal{A}_{X,Z}^{\bullet})  \ar[r]^{} &0 \ar[r] &0\\
 \ar[r]&\mathcal{H}^{1}(\mathcal{A}_{X,Z}^{\bullet}) \ar[r]^{} & 0 \ar[r]&0\\
 \ar[r]& \cdots& &\\
  \ar[r] &\mathcal{H}^{p}(\mathcal{A}_{X,Z}^{\bullet}) \ar[r] &0 \ar[r] &0\\
   \ar[r]& \cdots& &.}
$$
So $\mathcal{H}^{p}(\mathcal{A}_{X,Z}^{\bullet}) =0$, for  $p\geq 1$.
By Lemma \ref{lemma 2.2}, $\underline{\mathbb{R}}_{Z,i^{\ast}\theta} \cong i^\ast \underline{\mathbb{R}}_{X,\theta}$ and thus \eqref{2.8} implies that $ \underline{\mathbb{R}}_{X,\theta} \rightarrow i_{\ast}\underline{\mathbb{R}}_{Z,i^{\ast}\theta} \cong i_\ast i^\ast \underline{\mathbb{R}}_{X,\theta}$ is an epimorphism. Hence, $ \mathcal{H}^{0}(\mathcal{A}_{X,Z}^{\bullet}) = 0. $

In summary, we have the short exact sequence
\begin{equation}\label{2.5}
\xymatrix{
0 \ar[r]^{} & \underline{\mathbb{R}}_{X,Z,\theta} \ar[r]& \underline{\mathbb{R}}_{X,\theta} \ar[r] &i_{\ast}\underline{\mathbb{R}}_{Z,i^{\ast}\theta} \ar[r]^{} &0,}
\end{equation}
and the resolution of $\underline{\mathbb{R}}_{X,Z,\theta}$
$$
\xymatrix{
0 \ar[r]^{} & \underline{\mathbb{R}}_{X,Z,\theta} \ar[r] & \mathcal{A}^{0}_{X,Z}  \ar[r]^{d_\theta} & \mathcal{A}^{1}_{X,Z} \ar[r]^{d_\theta} & \cdots}
$$
by $\mathcal{H}^{p}(\mathcal{A}_{X,Z}^{\bullet}) =0$, for  $p\geq 0$. We know that $ \mathcal{A}^{p}_{X,Z} $ is a ${C}_X^{\infty}$-module where ${C}_X^{\infty}$ is the sheaf of germs of ${C}^{\infty}$ differentiable functions over $X$. Since ${C}_X^{\infty}$ is a fine sheaf, $\mathcal{A}^{p}_{X,Z}$ is a fine sheaf too. So one can calculate the cohomology of $ \underline{\mathbb{R}}_{X,Z,\theta} $ by
\begin{equation} \label{2.7}
 H^{p}(X,~\underline{\mathbb{R}}_{X,Z,\theta}) = H^{p}(\Gamma(X,~\mathcal{A}^{\bullet}_{X,Z}), d_{\theta}) = H_{\theta}^{\bullet}(X,Z).
\end{equation}

Now let us introduce some definitions. For a sheaf $\mathfrak{F}$ on a topological space $X$ and a section $s$ of $\mathfrak{F}$ over the open subset $U$ of $X$, we define the support $Supp(s)$ of $s$ to be
$$ Supp(s) = \{ x\in U : s_x \neq 0\}.
$$
Consider a locally closed subspace $ W $ of $X$. This means that any point of  $W$ has an open neighbourhood $V$ in $X$ such that $W\cap V$ is closed relative to $V$. The inclusion of $W$ in $X$ is denoted by $h: W \rightarrow X.$

\begin{defn}
For a sheaf $\mathfrak{E}$ on $W$, $h_{!}\mathfrak{E}$ denotes the sheaf on $X$ whose sections over an open set $U$ of $X$ is given by
$$  \Gamma(U,~h_{!}\mathfrak{E}) = \{s \in \Gamma(W\cap U,~\mathfrak{E}) : \text{Supp(s) ~is~ closed~ relative~ to~ U}\}.
$$
\end{defn}

 Then we turn to the complex manifold $X$. As usual, set $U:=X-Z$ and let $ j: U\rightarrow X$ be the inclusion. A sheaf $\mathfrak{F}$ on $X$ gives rise to an exact sequence,
\begin{equation} \label{2.8}
\xymatrix{
0 \ar[r]^{} & j_!j^{\ast}\mathfrak{F} \ar[r]& \mathfrak{F} \ar[r] & i_{\ast}i^{\ast}\mathfrak{F} \ar[r]^{} &0.}
\end{equation}
According to Lemma \ref{lemma 2.2}, one has $ \underline{\mathbb{R}}_{Z,i^{\ast}\theta} = i^{\ast} \underline{\mathbb{R}}_{X,\theta} $. The sequences \eqref{2.5} and \eqref{2.8} yield
$$ \underline{\mathbb{R}}_{X,Z,\theta} =  j_!j^{\ast}\underline{\mathbb{R}}_{X,\theta}.$$
Hence, $H^{p}(X,~\underline{\mathbb{R}}_{X,Z,\theta})$ is precisely the relative cohomology $H^{p}(X,Z;~\underline{\mathbb{R}}_{X,\theta})$ defined in the \cite[Definition IV.8.1]{iv} and \cite[Proposition II.12.3]{Br} since
\begin{equation} \label{2.9}
H^{p}(X,~\underline{\mathbb{R}}_{X,Z,\theta}) \cong H^{p}(X,~j_!j^{\ast}\underline{\mathbb{R}}_{X,\theta}) \cong H^{p}(X,Z;~\underline{\mathbb{R}}_{X,\theta}).
\end{equation}

Finally, we apply the definition of \emph{relative homeomorphism} as in \cite{Br}, that is a closed map $f$ of pairs $(X, A) \rightarrow (Y, B)$ such that $ A= f^{-1}B $ and the induced  map $ X-A\rightarrow Y-B $ is a homeomorphism.

\begin{lem} [{\cite[Corollary 12.5]{Br}}]  \label{lemma 2.4}
If $(X,A)$ and $(Y,B)$ are two closed paracompact pairs and $ f: (X,A) \rightarrow (Y,B)$ is a relative homeomorphism, then  for any sheaf $ \mathfrak{A}$ on Y,
$$ f^{\ast}: H^{\ast}(Y,B;\mathfrak{A}) \rightarrow H^{\ast}(X,A;f^{\ast}\mathfrak{A}) $$
is an isomorphism.
\end{lem}

\section{Proof of Main Theorem \ref{main-thm}}\label{pmt}
We combine the relative cohomological method in \cite{ryy,ryy2,ryyy} with the sheaf theory of \cite{meng2} induced as above, to give a new simple proof of the blow-up formula for Morse--Novikov cohomology with an explicit isomorphism.
\subsection{Relative cohomological method}
From now on, we set $U:=X-Z$ and let $i: Z \hookrightarrow X$ and $ j: U\rightarrow X$ be the inclusions.
Let $\{\mathcal{A}^{\bullet}(X,Z), d_\theta\}$ be the \emph{relative Morse--Novikov complex}.
Then we obtain a short exact sequence
$$\label{mainexactseq1}
\xymatrix@C=0.5cm{
0 \ar[r]^{} & \mathcal{A}^{\bullet}(X,Z) \ar[r]^{} & \mathcal{A}^{\bullet}(X)  \ar[r]^{i^{\ast}} & \mathcal{A}^{\bullet}(Z) \ar[r]^{} & 0}
$$
and analogously,
$$\label{mainexactseq2}
\xymatrix@C=0.5cm{
0 \ar[r]^{} & \mathcal{A}^{\bullet}(\tilde{X},E) \ar[r]^{} & \mathcal{A}^{\bullet}(\tilde{X})  \ar[r]^{\tilde{i}^{\ast}} & \mathcal{A}^{\bullet}(E) \ar[r]^{} & 0.}
$$
In particular, the blow-up diagram \eqref{blow-up} induces a commutative diagram of short exact sequences
\begin{equation}\label{3.1}
\xymatrix@C=0.5cm{
 0 \ar[r]^{} & \mathcal{A}^{\bullet}(X,Z) \ar[d]_{\pi^{* }} \ar[r]^{} & \mathcal{A}^{\bullet}(X) \ar[d]_{\pi^{* }} \ar[r]^{i^{* }} &  \mathcal{A}^{\bullet}(Z)\ar[d]_{\pi_{E}^*} \ar[r]^{} & 0 \\
0 \ar[r] &  \mathcal{A}^{\bullet}(\tilde{X},E) \ar[r]^{} &
   \mathcal{A}^{\bullet}(\tilde{X}) \ar[r]^{\tilde{i}^{* }} &
   \mathcal{A}^{\bullet}(E) \ar[r] &0. }
\end{equation}
Then the commutative diagram \eqref{3.1} gives a commutative ladder of long exact sequences
\begin{equation}\label{3.2}
\xymatrix@C=0.5cm{
   \cdots \ar[r]^{} & H^{k}_{\theta}(X,Z) \ar[d]_{\pi^{* }} \ar[r]^{} &H^{k}_{\theta}(X) \ar[d]_{\pi^{* }} \ar[r]^{} & H^{k}_{i^\ast\theta}(Z)\ar[d]_{\pi_{E}^*} \ar[r]^{} & H^{k+1}_{\theta}(X,Z)\ar[d]_{\pi^{* }} \ar[r]^{} & \cdots \\
   \cdots \ar[r] & H^{k}_{\pi^\ast\theta}(\tilde{X},E) \ar[r]^{} &
  H^{k}_{\pi^\ast\theta}(\tilde{X}) \ar[r]^{} &
  H^{k}_{\pi_E^\ast i^\ast\theta}(E)\ar[r]^{} &
  H^{k+1}_{\pi^\ast\theta}(\tilde{X},E) \ar[r] & \cdots. }
\end{equation}

\begin{prop}
$ \pi^\ast: H_\theta^k(X,Z) \rightarrow H_{\pi^\ast\theta}^k(\widetilde{X},E) $ is isomorphic.
\end{prop}
\begin{proof}
It is obtained directly from Lemma \ref{lemma 2.2}, the equation \eqref{2.7}, the equation \eqref{2.9} and Lemma \ref{lemma 2.4}.
\end{proof}

One needs a proposition of Meng \cite[Corollary 2.9]{meng2} by the projection formula and here we give another proof following \cite{Wells74}, which is to be postponed in the next section.

\begin{prop} \label{prop 3.2}
Let $ \pi: \widetilde{X}\rightarrow X $ be a proper smooth map of oriented connected smooth manifolds with the same dimension and $\emph{deg}\ \pi\neq 0$. If $\theta$ is a closed $1$-form on $X$, then $\pi^*:H^*_{\theta}(X)\rightarrow H^*_{\pi^\ast{\theta}}(\widetilde{X})$ is injective.
\end{prop}
Note that since
$\pi:\tilde{X}\rightarrow X$ and $\pi_{E}:E\rightarrow Z$ are proper surjective holomorphic maps, the pullback $\pi^{\ast}:H^{k}_{\theta}(X)\rightarrow H^{k}_{\pi^\ast\theta}(\tilde{X})$ and $ \pi_{E}^{\ast}:H^{k}_{i^\ast\theta}(Z)\rightarrow H^{k}_{\pi_E^\ast i^\ast\theta}(E)$ are injective. The next proposition is crucial to build the isomorphism of main theorem.

\begin{prop}[{\cite[Proposition 3.3]{YZ15}}]\label{cmm-diag}
 Consider a commutative ladder of abelian groups such that its horizontal rows are exact
\begin{equation*}
\xymatrix@C=0.5cm{
  \cdots \ar[r]^{} & A_1 \ar[d]_{i_1} \ar[r]^{f_1}& A_2 \ar[d]_{i_2} \ar[r]^{f_2}& A_3 \ar[d]_{i_3} \ar[r]^{f_3}& A_4 \ar[d]_{i_4} \ar[r]^{f_4}& A_5 \ar[d]_{i_5} \ar[r]^{} & \cdots \\
\cdots \ar[r]^{} & B_1 \ar[r]^{g_1}& B_2 \ar[r]^{g_2}& B_3 \ar[r]^{g_3}& B_4 \ar[r]^{g_4}& B_5   \ar[r]^{} & \cdots. }
\end{equation*}
Assume that $i_1$ is epimorphic, $i_2,i_3,i_5$ are monomorphic and $i_4$ is isomorphic.
Then there holds a natural isomorphism
$$
\mathrm{coker}\,i_2\cong\mathrm{coker}\,i_3.
$$
\end{prop}

Based on the Proposition \ref{cmm-diag}, one can get an isomorphism from commutative ladder \eqref{3.2}
\begin{equation}\label{3.3}
{H^{k}_{\pi^\ast\theta}(\tilde{X})}/{\pi^{\ast}H^{k}_{\theta}(X)}\cong
{H^{k}_{\pi_E^\ast i^\ast\theta}(E)}/{\pi_{E}^{\ast}H^{k}_{i^\ast\theta}(Z)}.
\end{equation}

Now we need to figure out the relationship between $ H^{k}_{\pi_E^\ast i^\ast\theta}(E) $ and $ H^{k}_{i^\ast\theta}(Z) $. Recall $E:=\pi^{-1}(Z)\cong \mathbb{P}(\mathcal{N}_{Z/X})$. Then we need the projective vector bundle case of Leray--Hirsch theorem. Here we refer to \cite{ryy} and \cite{ryy2} for the Borel spectral sequence approach.

\begin{prop} [{\cite[Lemma 4.4]{meng2}}] \label{proj-bundle}
Let $\pi:\mathbb{P}(E)\rightarrow X$ be the projectivization of a holomorphic vector bundle $E$ of rank $r$ on a complex manifold $X$ and $\theta$ a closed $1$-form on $X$. Assume that $\tilde{\theta}=\pi^*\theta$ and $h=c_1(\mathcal{O}_{\mathbb{P}(E)}(-1))\in H_{dR}^2({\mathbb{P}(E)})$ is the first Chern class of the universal line bundle $\mathcal{O}_{\mathbb{P}(E)}(-1)$ on ${\mathbb{P}(E)}$. Then $\pi^*(\bullet)\wedge\bullet$ gives isomorphisms of graded vector spaces
\begin{displaymath}
H_{\theta}^*(X)\otimes_{\mathbb{R}}\emph{span}_{\mathbb{R}}\{1,\cdots,h^{r-1}\}\tilde{\rightarrow}H_{\tilde{\theta}}^*(\mathbb{P}(E)).
\end{displaymath}
\end{prop}

According to Proposition \ref{proj-bundle} and \eqref{excep}, one gets the isomorphism
\begin{equation} \label{3.4}
{H^{k}_{\pi_E^\ast i^\ast\theta}(E)}/{\pi_{E}^{\ast}H^{k}_{i^\ast\theta}(Z)}\cong
\bigoplus_{j=1}^{r-1}H^{k-2j}_{i^\ast\theta}(Z).
\end{equation}
Combing \eqref{3.3} and \eqref{3.4} and observing that $\pi^{\ast}$ is injective, one can obtain the Morse--Novikov blow-up isomorphic formula
$$
H^{k}_{\pi^\ast\theta}(\tilde{X})
\cong H^{k}_{\theta}(X)\oplus \left(\bigoplus_{j=1}^{r-1}H^{k-2j}_{i^\ast\theta}(Z) \right).
$$
This completes the proof for the isomorphism part of Main Theorem \ref{main-thm}.
\subsection*{Remark}
In \cite{meng3}, Meng studies the cohomologies with values in local system on smooth manifolds. Our approach should be applied to the blow-up formula for local system with finite rank.

\subsection{Proof of Proposition \ref{prop 3.2}}
First we assume that $X$ is an oriented differentiable manifold of dimension $n$. Let $\mathcal{A}_c^{k}(X)$ be the space of compactly supported differential forms in $\mathcal{A}^{k}(X)$. We define the vector space $\mathcal{D}^{k}(X)$ of currents of type $k$ on $X$ as the dual space of the topological vector space $\mathcal{A}_c^{n-k}(X)$. And let $\mathcal{D}_X^{k}$ be the sheaf of germs of $k$-currents. Similarly, define $\textrm{d}_\theta:\mathcal{D}_X^{k}\rightarrow\mathcal{D}_X^{k+1}$ as follows:
\begin{displaymath}
\textrm{d}_\theta T=\textrm{d}T+\theta\wedge T,
\end{displaymath}
for $T\in\mathcal{D}_X^{k}$. A direct check shows $$\textrm{d}_{\theta}T(\alpha)=(-1)^{k+1}T(\textrm{d}_{-\theta}\alpha),$$ for $\alpha\in\mathcal{A}_c^{n-k-1}(X)$. There is another  resolution of soft sheaves of $\underline{\mathbb{R}}_{X,\theta}$
\begin{displaymath}
\xymatrix{
0\ar[r] &\underline{\mathbb{R}}_{X,\theta}\ar[r]^{i} &\mathcal{D}_X^{0}\ar[r]^{\textrm{d}_\theta} &\mathcal{D}_X^{1}\ar[r]^{\textrm{d}_\theta}&\cdots\ar[r]^{\textrm{d}_\theta}&\mathcal{D}_X^{n}\ar[r]&0
,}
\end{displaymath}
where $i$ is the inclusion. Moreover, $\mathcal{A}_X^\bullet\hookrightarrow\mathcal{D}_X^{\bullet}$ induces isomorphisms $$H_\theta^*(X)\tilde{\rightarrow} H^*_{\mathscr{D},\theta}(X):= H^*(\mathcal{D}^{\bullet}(X),\textrm{d}_{\theta}).$$

We now turn to general differentiable manifolds. Suppose that $ \pi: \widetilde{X}\rightarrow X $ is a surjective proper differentiable mapping of orientable differentiable manifolds of the same dimension. Then one has the ladder
\begin{equation*}
\xymatrix@C=0.5cm{
  \cdots \ar[r]^{} & \mathcal{A}^k(\widetilde{X}) \ar[r]^{d_{\pi^\ast\theta}}& \mathcal{A}^{k+1}(\widetilde{X}) \ar[r]^{} & \cdots \\
\cdots \ar[r]^{} & \mathcal{A}^k(X)\ar[u]_{\pi^\ast} \ar[r]^{d_\theta}& \mathcal{A}^{k+1}(X)\ar[u]_{\pi^\ast} \ar[r]^{} & \cdots, }
\end{equation*}
and by duality, there is a push-out $ \pi_{\flat} $ of the dual ladder
\begin{equation*}
\xymatrix@C=0.5cm{
  \cdots \ar[r]^{} & \mathcal{D}^{k}(\widetilde{X}) \ar[d]_{\pi_{\flat}} \ar[r]^{d_{\pi^\ast\theta}}& \mathcal{D}^{k+1}(\widetilde{X}) \ar[d]_{\pi_{\flat}} \ar[r]^{} & \cdots \\
\cdots \ar[r]^{} & \mathcal{D}^{k}(X) \ar[r]^{d_\theta}& \mathcal{D}^{k+1}(X) \ar[r]^{} & \cdots. }
\end{equation*}
Now consider the diagram
\begin{equation} \label{3.5}
\xymatrix@C=0.5cm{
 & \mathcal{A}^k(\widetilde{X}) \ar[r]^{\widetilde{\varrho}}& \mathcal{D}^{k}(\widetilde{X}) \ar[d]_{\pi_{\flat}} \\
 & \mathcal{A}^k(X)\ar[u]_{\pi^\ast} \ar[r]^{\varrho}& \mathcal{D}^{k}(X), }
\end{equation}
where $ \varrho $ and $ \widetilde{\varrho} $ are the natural injections.

\begin{lem} [{\cite[Lemma 2.1]{Wells74}}] \label{lemma 3.5}
In the above diagram \eqref{3.5}, $\mu \varrho = \pi_\ast\widetilde{\varrho}\pi^\ast $ where $ \mu $ is the degree of the mapping $ \pi $.
\end{lem}
In other words, the diagram \eqref{3.5} is commutative up to a fixed constant multiple, which would not affect the passage to cohomology later on.
According to Lemma \ref{3.5}, we obtain the commutative diagram
\begin{equation} \label{3.6}
 \xymatrix{
  H_{\pi^\ast\theta}^k(\widetilde{X}) \ar[r]^{\widetilde{\varrho}_\ast}
                & H_{\mathscr{D}, \pi^\ast\theta}^k(\widetilde{X}) \ar[d]^{\pi_{\flat}}  \\
  H_\theta^k(X) \ar[u]_{\pi^\ast}\ar[r]^{\varrho_\ast}
                & H_{\mathscr{D}, \theta}^k(X)   .          }
\end{equation}
Now we can prove the injection of $\pi^\ast$. If $\xi\in H_\theta^k(X) $ such that $ \pi^\ast\xi=0 $, then $ \pi_{\flat}\widetilde{\varrho}_\ast\pi^\ast\xi = \mu \varrho_\ast\xi = 0 $ and since $ \mu\neq0 $ and $ \varrho_\ast $ is injective, it follows that $ \xi=0 $. This completes the proof of Proposition \ref{prop 3.2}.

\subsection{Explicit isomorphisms between cohomologies}
The construction of the morphism $\phi$ is inspired by \cite{ryy2}. In view of Proposition \ref{proj-bundle}, we get
$$
H^{k}_{\pi_E^\ast i^\ast\theta}(E) \cong
\bigoplus^{r-1}_{j=0}\bm{h}^{j}\wedge
\pi_{E}^{\ast}H^{k-2j}_{i^\ast\theta}(Z),
$$
where
$
\bm{h}:=c_1(\mathcal{O}_{E}(-1))\in H_{dR}^2(E),
$
and
$\pi_{E}^{\ast}$ is the pullback of the projection
$\pi_{E}: E\rightarrow Z$.
It means that each class
$[\tilde{\alpha}]_{k}\in H^{k}_{\pi_E^\ast i^\ast\theta}(E)$
admits a unique expression
$$
[\tilde{\alpha}]_{k}=
\sum^{r-1}_{j=0}\bm{h}^{j}\wedge
\pi_{E}^{\ast}[\alpha]_{k-2j},
$$
where
$[\alpha]_{k-2j} \in H^{k-2j}_{i^\ast\theta}(Z)$.
It's natural to define the linear map
\begin{eqnarray*}
  \Pi_{j}:H^{k}_{\pi_E^\ast i^\ast\theta}(E)
  &\rightarrow& H^{k-2j}_{i^\ast\theta}(Z) \\
  {[\tilde{\alpha}]}_{k}&\mapsto& {[\alpha]}_{k-2j}.
\end{eqnarray*}
Then we can define the desired morphism $\phi$ by setting
\begin{equation}\label{phi}
\phi=\pi_{\ast}+\sum_{j=1}^{r-1}\Pi_{j}\circ \tilde{i}^{\ast}
\end{equation}
which maps
$\mathbb{V}^{k}:=H^{k}_{\pi^\ast\theta}(\tilde{X})$
to the space
$$
\mathbb{W}^{k}:=
 H^{k}_{\theta}(X)\oplus \Big(\bigoplus_{j=1}^{r-1} H^{k-2j}_{i^\ast\theta}(Z) \Big).
$$

 Over a compact complex manifold, the Morse--Novikov cohomology group has finite dimension. Since $\mathbb{V}^{k}$ is isomorphic to $\mathbb{W}^{k}$ as proved, it suffices  to verify the injectivity of $\phi$.
From \eqref{3.2}, one obtains a commutative diagram of short exact sequences
\begin{equation}\label{comm-coker}
\xymatrix@C=0.5cm{
0 \ar[r]^{} & H^{k}_{\theta}(X)
\ar[d]_{i^{\ast}}
\ar[r]^{\pi^{\ast}} &
H^{k}_{\pi^\ast\theta}(\tilde{X})
\ar[d]_{\tilde{i}^{\ast}}
\ar[r]^{} &
\mathrm{coker}\,(\pi^{\ast}) \ar[d]_{\tilde{j}^\ast}^{\cong} \ar[r]^{} & 0 \\
0 \ar[r] &
H^{k}_{i^\ast\theta}(Z)
\ar[r]^{\pi_{E}^{\ast}} &
H^{k}_{\pi_{E}^\ast i^\ast\theta}(E)
\ar[r]^{} &
\mathrm{coker}\,(\pi_{E}^\ast)\ar[r]^{} & 0.}
\end{equation}
Here $\tilde{j}^{\ast}$ is the induced isomorphism of the quotient spaces by $\tilde{i}^{\ast}$.
Combining \eqref{3.6} with \eqref{comm-coker} gives rise to the following diagram
$$\label{curr-coker}
\xymatrix@C=0.5cm{
& H^{k}_{\mathscr{D}, \theta}(X)
 &
H^{k}_{\mathscr{D}, \pi^\ast\theta}(\tilde{X})
\ar[l]_{\pi_{\flat}}
&
& \\
0 \ar[r]^{} & H^{k}_{\theta}(X)
\ar[u]^{\varrho_{\ast}}_{\cong}
\ar[d]_{i^{\ast}}
\ar[r]^{\pi^{\ast}} &
H^{k}_{\pi^\ast\theta}(\tilde{X})
\ar[d]_{\tilde{i}^{\ast}}
\ar[u]^{\tilde{\varrho}_{\ast}}_{\cong}
\ar[r]^{} &
\mathrm{coker}\,(\pi^{\ast}) \ar[d]_{\tilde{j}^\ast}^{\cong} \ar[r]^{} & 0 \\
0 \ar[r] &
H^{k}_{i^\ast\theta}(Z)
\ar[r]^{\pi_{E}^{\ast}} &
H^{k}_{\pi_{E}^\ast i^\ast\theta}(E)
\ar[r]^{} &
\mathrm{coker}\,(\pi_{E}^\ast)\ar[r]^{} & 0.}
$$
Recall that
$
\pi_{\ast}=
\varrho^{-1}_{\ast}\circ\pi_{\flat}
\circ\tilde{\varrho}_{\ast}
$
and
$
\phi=\pi_{\ast}+\sum_{j=1}^{r-1}\Pi_{j}\circ \tilde{i}^{\ast}
$
maps $H^{k}_{\pi^\ast\theta}(\tilde{X})$ to
$$
H^{k}_{\theta}(X)\oplus \Big(\bigoplus_{j=1}^{r-1} H^{k-2j}_{i^\ast\theta}(Z) \Big).
$$
We already know that $\pi_{\ast}$ is surjective and a direct check shows that
the kernel of $\pi_{\ast}$ is equal to the co-kernel of $\pi^{\ast}$, i.e.,
$$
\ker\,(\pi_{\ast})=\mathrm{coker}\,(\pi^{\ast})
\stackrel{\tilde{j}^{\ast}}\rightarrow
\biggl(\bigoplus^{r-1}_{j=1}\bm{h}^{j}\wedge
\pi_{E}^{\ast}H^{k-2j}_{i^\ast\theta}(Z)\biggr).
$$
It follows that the restriction of $\tilde{i}^{\ast}$ on
$\ker\,(\pi_{\ast})$ is injective.
Given an element
$\tilde{\alpha}\in H^{k}_{\pi^\ast\theta}(\tilde{X})$, suppose $\phi(\tilde{\alpha})=0$.
Then we get
$\tilde{\alpha}\in\ker\,(\pi_{\ast})$,
$\tilde{i}^{\ast}(\tilde{\alpha})=0$ and so $\tilde{\alpha}=0$.
This implies that $\phi$ is injective.

We just proved the isomorphism of $\phi$ when the manifold $X$ is compact.  And in non-compact case, it holds too. In fact, the Morse--Novikov cohomology can be calculated by the sheaf cohomology of $\mathbb{R}_{X,\theta}$,
 $$ H_\theta^\ast(X) \cong H^\ast(X,\mathbb{R}_{X,\theta}). $$
 The weight $\theta$-sheaf $\underline{\mathbb{R}}_{X,\theta} $ is a locally constant sheaf of $ \mathbb{R}$-modules of rank $1$. According to \cite[Lemma 4.3]{meng2} and \cite[Proposition 6.8]{meng3}, the morphism $\phi$ and $\psi$ in \eqref{psi} are inverse to each other. Since Meng has proved that the morphism $\psi$ is isomorphic in \cite[Main Theorem 1.3]{meng2}, $\phi$ is also an isomorphism.

\end{document}